\documentclass[
final
]{dmtcs-episciences}

\usepackage[utf8]{inputenc}
\usepackage{amsmath}
\usepackage{url}
\usepackage{algorithm}
\usepackage{algpseudocode}

\algrenewcommand{\algorithmiccomment}[1]{\hfill {\tt // {\scriptsize #1 }}}

\usepackage{tikz}
\usetikzlibrary{
  calc,
  positioning
}

\tikzset{
  nodestyle/.style={
    draw,
    solid,
    color=black,
    circle,
    line width=1pt,
    minimum size=0.25cm,
    inner sep=1pt,
    text width=0.3cm,
    align=center,
    font=\scriptsize
  },
  linestyle/.style={
    color=black,
    line width=1.25pt,
    solid
  },
  every node/.style={nodestyle},
  every path/.style={linestyle}
}

\usepackage{graphicx}
\usepackage[T1]{fontenc}
\usepackage{geometry}
\usepackage{amsmath}
\usepackage{amsthm}

\newtheorem{thm}{Theorem}[section]
\newtheorem{lem}[thm]{Lemma}
\newtheorem{cor}[thm]{Corollary}
\newtheorem{defn}[thm]{Definition}

\newtheorem{conjecture}[thm]{Conjecture}

\newtheorem*{acknowledgement*}{Acknowledgement}

\usepackage{tikz}
\usepackage{tkz-graph}
\usepackage{tkz-berge}
\usepackage{tkz-euclide}
\usepackage{amsmath}
\usepackage{amsthm}
\usepackage{amssymb}
\usepackage{stmaryrd}
\usepackage{subcaption}

\author[Allan Bickle and Russell Campbell]{Allan Bickle\affiliationmark{1}
  \and Russell Campbell\affiliationmark{2}}
\title{Planar-Toroidal Decomposition of $K_{12}$}
\affiliation{
  Department of Mathematics,  Purdue University, West Lafayette,  USA\\
  School of Computing, University of the Fraser Valley, Abbotsford, Canada}
\begin{document}
\publicationdata{vol. 28:2}{2026}{3}{10.46298/dmtcs.16100}{2025-07-24; None}{2025-12-28}

\maketitle
\bigbreak

\begin{abstract}
  In 1978, Anderson and White asked whether there is a decomposition
of $K_{12}$ into two graphs, one planar and one toroidal. Using theoretical
arguments and a computer search of all maximal planar graphs of order
12, we show that no such decomposition exists. We further show that
if $G$ is planar of order 12 and $H\subseteq\overline{G}$ is toroidal,
then $H$ has at least two fewer edges than $\overline{G}$. A computer
search found all 123 unique pairs $\left(G,H\right)$ that make this
an equality.
\end{abstract}

\section{Introduction}

Many researchers have worked on problems involving splitting complete
graphs into subgraphs with various topological properties.
\begin{defn}
A \textbf{decomposition} of $G$ is a set of nonempty subgraphs, called
\textbf{factors}, whose edge sets partition $E\left(G\right)$. The
subgraphs are said to \textbf{decompose} $G$. The \textbf{thickness}
$\theta\left(G\right)$ of a graph $G$ is the minimum number of planar
graphs that decompose $G$.
\end{defn}

\begin{thm}
We have $\theta\left(K_{n}\right)=\left\lfloor \frac{n+7}{6}\right\rfloor $
unless $n\in\left\{ 9,10\right\} $, and $\theta\left(K_{9}\right)=\theta\left(K_{10}\right)=3$.
\end{thm}

The case of $K_{9}$ was settled by Harary, Battle and Kodoma \cite{BHK 1962}
and Tutte \cite{Tutte 1963B}. Tutte checked all 50 maximal planar
graphs of order 9, showing that none of their complements are planar.
A general construction showing $\theta\left(K_{n}\right)=\left\lfloor \frac{n+7}{6}\right\rfloor $
was found by Beineke and Harary \cite{Beineke/Harary 1965} for $n\not\equiv4\mod6$.
Finally, Vasok \cite{Vasok 1976} and Alekseev and Gonchakov \cite{Alekseev/Gonchakov 1976}
completed the proof of the theorem for $K_{6r+4}$.
\begin{defn}
A \textbf{torus} is a surface with one handle (or hole). A graph that
can be drawn on a torus with no crossings is \textbf{toroidal}.

The \textbf{genus of a surface} is the number of handles (or holes)
it has. The surface with $k$ handles is denoted $S_{k}$. An \textbf{embedding}
of a graph is a drawing with no crossings. The \textbf{genus of a
graph} $\gamma\left(G\right)$ is the minimum genus so that it has
an embedding on a surface with this genus. A region is a \textbf{2-cell}
if any closed curve within that region can be continuously contracted
to a point within that region.

The \textbf{$S_{k}$-thickness} $\theta_{k}\left(G\right)$ of a graph
$G$ is the minimum number of \textbf{$S_{k}$}-embeddable graphs
that decompose $G$.
\end{defn}

Beineke \cite{Beineke 1969} has shown that $\theta_{1}\left(K_{n}\right)=\left\lfloor \frac{n+4}{6}\right\rfloor $
and $\theta_{2}\left(K_{n}\right)=\left\lfloor \frac{n+3}{6}\right\rfloor $.
Thus the \textbf{toroidal thickness} $\theta_{1}\left(K_{12}\right)=2$
\cite{Beineke 1969}, meaning that $K_{12}$ can be decomposed into
two toroidal graphs.
\begin{defn}
A graph $G$ is $\left(\gamma,\gamma'\right)$ \textbf{bi-embeddable}
if $G$ can be embedded in $S_{\gamma}$, a sphere with $\gamma$
handles, and $\overline{G}$ can be embedded in $S_{\gamma'}$. Let
$N\left(\gamma,\gamma'\right)$ be the size of the smallest complete
graph which cannot be edge-partitioned into two parts embeddable in
$S_{\gamma}$, and $S_{\gamma'}$, respectively.
\end{defn}

The problem of finding upper and lower bounds for $N\left(\gamma,\gamma'\right)$
was first studied in 1974 by Anderson and Cook in \cite{Anderson/Cook 1974}.
Results on thickness imply that $N\left(0,0\right)=9$. The work of
Ringel \cite{Ringel 1965} and Beineke \cite{Beineke 1969} on toroidal
thickness showed that $N\left(1,1\right)=14$, and Beineke showed
that $N\left(2,2\right)=15$. Bi-embeddings were studied further by
Anderson \cite{Anderson 1979}, Cabaniss and Jackson \cite{Cabaniss/Jackson 1990}
and Sun \cite{Sun 2022}. Cabaniss \cite{Cabaniss 1990} surveyed
bi-embeddings, while Beineke \cite{Beineke 1997} surveyed biplanar
graphs.

In 1978, Anderson and White \cite{Anderson/White 1978} were the first
to explicitly raise the question of determining $N\left(0,1\right)$.
This can be accomplished by determining whether there is a maximal
planar graph of order 12 whose complement is maximal toroidal. We
will show that this is not possible. Along with the known decomposition
of $K_{11}$ into $\left\{ C_{9}+2K_{1},\overline{C}_{9}\cup K_{2}\right\} $,
this shows that $N\left(0,1\right)=12$.

In 2013, Bickle and White \cite{Bickle/White 2013} found bounds on
$\gamma\left(G\right)+\gamma\left(\overline{G}\right)$ and $\gamma\left(G\right)\cdot\gamma\left(\overline{G}\right)$.
In particular, they showed that $\gamma\left(G\right)+\gamma\left(\overline{G}\right)\geq\left\lceil \frac{1}{12}\left(n^{2}-13n+24\right)\right\rceil $,
and that this is attained for order $n=12s+11$ and for $n\in\left\{ 13,25,37,49\right\} $.
Sun \cite{Sun 2022} further showed this is attained for $n=24s+13$.
\begin{conjecture}
(Bickle/White \cite{Bickle/White 2013}) For all $n\geq11$, the lower
bound $\left\lceil \frac{1}{12}\left(n^{2}-13n+24\right)\right\rceil $
of $\gamma\left(G\right)+\gamma\left(\overline{G}\right)$ is attained
by some graph $G$.
\end{conjecture}

For $n=12$, this says that there is a graph $G$ for which $\gamma\left(G\right)+\gamma\left(\overline{G}\right)=1$.
That is, there is a decomposition of $K_{12}$ into two graphs, one
planar and one toroidal. Thus we have disproved the conjecture when
$n=12$.

Definitions of terms and notation not defined here appear in \cite{Bickle 2020}.
In particular, $n\left(G\right)$ is the number of vertices of a graph
$G$. The degree of a vertex $v$ is denoted $d_{G}\left(v\right)$,
or $d\left(v\right)$ when the graph in question is clear. The neighborhood
of a vertex $v$ is denoted $N\left(v\right)$, and the closed neighborhood
is denoted $N\left[v\right]$. The join of graphs $G$ and $H$ is
denoted $G+H$. A \textbf{separating set} of a connected graph $G$
is a set $S$ of vertices so that $G-S$ is disconnected.

\section{Theoretical Approaches}

We describe three theoretical approaches to checking whether a maximal
planar graph of order 12 has a complement that embeds on the torus.
Each approach eliminates many cases, but even together, they do not
eliminate all possibilities. We hope to eventually use the descriptions
to both form the basis of a non-computer proof, and provide yet another
way for collections of graphs that can be used as input to help design
efficient algorithms, specifically, using graphs whose factors embed
on the plane and the torus. These approaches may also be useful in
related problems, such as other values of $N\left(\gamma,\gamma'\right)$.

\subsection{Vertex Degrees}

Any planar graph has size $m\leq3n-6$, and any toroidal graph has
size $m\leq3n$. A complete graph has size $\binom{n}{2}$. For $n=12$,
these numbers are 30, 36, and 66. Thus if a plane-torus decomposition
existed, both factors have the maximum number of edges. For planar
graphs being maximal and being a triangulation are equivalent. Note
however that a maximal toroidal graph need not be a triangulation
\cite{HKSW 1974}.
\begin{defn}
We say an ordered pair of graphs $\left(G,\overline{G}\right)$ are
\textbf{PT12} if $G$ is planar, $\overline{G}$ is toroidal, and
$\left\{ G,\overline{G}\right\} $ decomposes $K_{12}$.
\end{defn}

\begin{lem}
If $\left(G,\overline{G}\right)$ are PT12, then $3\leq\delta\left(G\right)\leq5\leq\Delta\left(G\right)\leq8$
and $3\leq\delta\left(\overline{G}\right)\leq6\leq\Delta\left(\overline{G}\right)\leq8$.
\end{lem}

\begin{proof}
In a triangulation, every vertex of $G$ has degree at least 3 when
$n\geq4$. If $d_{G}\left(v\right)=d$, $d_{\overline{G}}\left(v\right)=n-1-d$,
Thus $\Delta\left(\overline{G}\right)\leq8$ and $\Delta\left(G\right)\leq8$.
It is well-known that for any planar graph, $\delta\left(G\right)\leq5$.
Thus $\Delta\left(\overline{G}\right)\geq6$.
\end{proof}

There are 7595 maximal planar graphs with order 12, which are listed
on the Combinatorial Object Server \cite{COS}. Of these graphs, there
are 3476 with $\Delta>8$.
\begin{lem}
\label{lem:83}If $\left(G,\overline{G}\right)$ are PT12, and $d_{G}\left(v\right)=8$,
then the 3 vertices not in $N_{G}\left[v\right]$ form an independent
set. The degree 8 vertices of $G$ are all adjacent.
\end{lem}

\begin{proof}
If $d_{G}\left(v\right)=8$, $d_{\overline{G}}\left(v\right)=3$.
Since $\overline{G}$ is a triangulation, the neighbors of $v$ in
$\overline{G}$ induce a triangle. These vertices must form an independent
set in $G$. The degree 3 vertices in $\overline{G}$ must all be
nonadjacent, so the degree 8 vertices of $G$ are all adjacent.
\end{proof}

\subsection{Counting Triangles}

A maximal planar graph with order $n$ has $2n-4$ triangular regions.
A toroidal triangulation with order $n$ has $2n$ triangular regions.
If $\left(G,\overline{G}\right)$ are PT12, then $\overline{G}$ has
24 triangular regions. Thus if $\overline{G}$ has fewer than 24 non-separating
triangles, it can be discarded.

For example, the icosahedron $IC$ is maximal planar with order 12.
It has 20 independent sets of size 3, so $\overline{IC}$ has 20 triangles.
Thus $\overline{IC}$ is not toroidal (this is stated without proof
in \cite{Anderson/White 1978}).

Note that a maximal planar or toroidal graph may have other triangles
that are not the boundaries of regions. Any vertex of degree 3 has
its neighbors induce a triangle. More generally, identifying two maximal
planar graphs on a triangle produces another maximal planar graph,
and identifying a maximal planar graph and a toroidal triangulation
produces a toroidal triangulation. A \textbf{separating triangle}
forms a cutset of $G$ or $\overline{G}$.

A toroidal triangulation may also have triangles that are not 2-cell.
A \textbf{handle triangle} cannot be contracted to a point in the
torus. For example, the toroidal embedding of $K_{7}$ has $\binom{7}{3}=35$
triangles, 14 regions and 21 handle triangles.

Intuitively, it appears that graphs that are close to regular have
fewer triangles in their complements.
\begin{thm}
(Goodman \cite{Goodman 1959}, Sauv\'e \cite{Sauve 1961}) Let
$t\left(G\right)$ be the number of triangles of $G$. Then 
\[
t\left(G\right)+t\left(\overline{G}\right)=\binom{n}{3}-\left(n-2\right)m+\sum_{v}\binom{d\left(v\right)}{2}.
\]
\end{thm}

A short proof of this is due to Schwenk \cite{Schwenk 1972}. Since
$t\left(G\right)+t\left(\overline{G}\right)$ depends only on the
degree sequence of $G$, some degree sequences can be ruled out immediately.
We use $d^{r}$ to indicate $r$ vertices of degree $d$.
\begin{cor}
If $\left(G,\overline{G}\right)$ are PT12, then $G$ does not have
degree sequence $5^{12}$, $4^{1}5^{10}6^{1}$, $4^{2}5^{8}6^{2}$,
$4^{3}5^{6}6^{3}$, $3^{1}5^{9}6^{2}$, $4^{2}5^{9}7^{1}$, or $3^{1}5^{10}7^{1}$.
\end{cor}

\begin{proof}
Since $G$ is maximal planar it has size $m=3n-6$. When $n=12$,
Goodman's formula becomes

\[
t\left(G\right)+t\left(\overline{G}\right)=-80+\sum_{v}\binom{d\left(v\right)}{2}.
\]
Since $t\left(G\right)\geq20$ and $t\left(\overline{G}\right)\geq24$,
we need $\sum_{v}\binom{d\left(v\right)}{2}\geq124$. The first six
sequences all have $\sum_{v}\binom{d\left(v\right)}{2}<124$. If $G$
has degree sequence $3^{1}5^{10}7^{1}$, $\sum_{v}\binom{d\left(v\right)}{2}=124$,
but $t\left(G\right)\geq21$ since $G$ must also have a separating
triangle due to the degree 3 vertex.
\end{proof}

\subsection{Separating Sets}

A separating set in $G$ implies that $\overline{G}$ contains a complete
bipartite graph. This observation provides a reasonable starting point
to eliminate large numbers of cases.
\begin{lem}
\label{lem:435} If $\left(G,\overline{G}\right)$ are PT12, then
$G$ has no separating triangle whose deletion leaves subgraphs with
orders 4 and 5.
\end{lem}

\begin{proof}
If not, then $\overline{G}$ contains $K_{4,5}$, which has genus
2 \cite{White 2001}.
\end{proof}
\begin{thm}
\label{thm:336} If $\left(G,\overline{G}\right)$ are PT12, then
$G$ has no separating triangle whose deletion leaves subgraphs with
orders 3 and 6.
\end{thm}

\begin{proof}
Assume to the contrary that $\left(G,\overline{G}\right)$ are PT12
and $G$ has a separating triangle whose deletion leaves subgraphs
with orders 3 and 6. Let $A$ be the set of vertices of the separating
triangle, $B$ be the set of 3 vertices, and $C$ be the set of 6
vertices. Now $B$ and $C$ induce $K_{3,6}$ in $\overline{G}$,
which has an embedding on the torus with every region a 4-cycle.

Now $A$ and $B$ induce a maximal planar subgraph $H$ of $G$ with
order 6. This must be either $K_{2,2,2}$ or $K_{2}+P_{4}$. There
are three different nonequivalent regions in $K_{2}+P_{4}$ that can
be $A$, but in each case some vertex $v$ of $A$ is adjacent to
all vertices in $B$. Then in $\overline{G}$, $v$ is adjacent to
no vertex of $B$. But this is impossible, since every region of $K_{3,6}$,
even when triangulated, contains a vertex of $B$. Thus $H=K_{2,2,2}$.
Thus any vertex $v\in A$ is adjacent to exactly one vertex of $B$
in $\overline{G}$.

Since the vertices of $A$ are all adjacent in $G$, they are all
nonadjacent in $\overline{G}$. Since $\overline{G}$ is a triangulation,
each (4-cycle) region of $K_{3,6}$ on the torus must have an edge
between two vertices of $C$. Adding these edges thus produces $K_{3,3,3}$.

Now in $\overline{G}$, each vertex in $A$ is adjacent to exactly
one vertex of $B$ and exactly two vertices of $C$. Each pair of
vertices in $C$ adjacent to a vertex of $A$ are also adjacent to
each other in $\overline{G}$, so nonadjacent in $G$.

Let $A=\left\{ u,v,w\right\} $. Each vertex of $A$ is adjacent in
$G$ to exactly four vertices of $C$. We seek to show that there
is no planar graph on $A\cup C$ satisfying the established conditions.

If any pair of vertices of $A$ (say $u$ and $v)$ has at least three
common neighbors in $C$, then two of them must be adjacent, and not
adjacent to $w$, contradicting Lemma \ref{lem:83}.

Suppose some pair of vertices of $A$ (say $u$ and $v)$ has exactly
two common neighbors in $C$, which must be adjacent. Then one of
them (call it $x$) must be adjacent to $w$. Now the triangles $uwx$
and $vwx$ have three neighbors of $w$ in their interiors. Thus one
(say $uwx$) contains at least two such neighbors, and hence two adjacent
vertices both not adjacent to $v$. This contradicts Lemma \ref{lem:83}.

If each pair of vertices in $A$ has exactly one common neighbor in
$C$, then $C$ contains at least 9 vertices, a contradiction.
\end{proof}

Next we consider deleting four vertices to produce two disjoint sets
of four vertices.
\begin{thm}
If $\left(G,\overline{G}\right)$ are PT12, then $G$ has no separating
4-cycle whose deletion leaves two subgraphs with order 4.
\end{thm}

\begin{proof}
Assume to the contrary that $\left(G,\overline{G}\right)$ are PT12
and $G$ has a separating 4-cycle $C=uvwxu$ whose deletion leaves
two (not necessarily connected) subgraphs with order 4. Let $A=\left\{ a,b,g,h\right\} $
and $B=\left\{ c,d,e,f\right\} $ be the two sets of 4 vertices. Now
$A$ and $B$ induce $K_{4,4}$ in $\overline{G}$, which can only
be embedded on the torus with 8 regions of length 4 (see the figure
below).

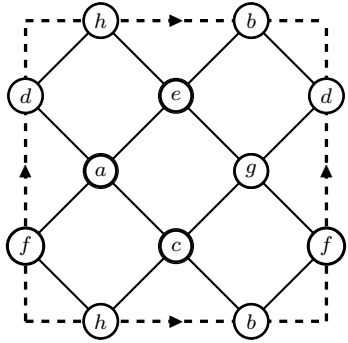
\begin{center}
\begin{tikzpicture}  
\SetVertexSimple[Shape=circle, MinSize=10pt, FillColor=white!50]
\draw[dashed] (0,0) -- (0,4);\node[draw=none] at (0,2) {$\blacktriangle$}; 
\draw[dashed] (4,0) -- (4,4);\node[draw=none] at (4,2) {$\blacktriangle$};
\draw[dashed] (0,0) -- (4,0);\node[draw=none] at (2,0) {$\blacktriangleright$};
\draw[dashed] (0,4) -- (4,4);\node[draw=none] at (2,4) {$\blacktriangleright$};
\Vertex[x=0,y=1]{01}\node at (0,1) {$f$};
\Vertex[x=0,y=3]{03}\node at (0,3) {$d$};
\Vertex[x=1,y=0]{10}\node at (1,0) {$h$};
\Vertex[x=1,y=2]{12}\node at (1,2) {$a$};
\Vertex[x=1,y=4]{14}\node at (1,4) {$h$};
\Vertex[x=2,y=1]{21}\node at (2,1) {$c$};
\Vertex[x=2,y=3]{23}\node at (2,3) {$e$};
\Vertex[x=3,y=0]{30}\node at (3,0) {$b$};
\Vertex[x=3,y=2]{32}\node at (3,2) {$g$};
\Vertex[x=3,y=4]{34}\node at (3,4) {$b$};
\Vertex[x=4,y=1]{41}\node at (4,1) {$f$};
\Vertex[x=4,y=3]{43}\node at (4,3) {$d$};
\Edges(21,32,43,34,23,12,01,10,21,30,41,32,23,14,03,12,21)
\end{tikzpicture}
\par\end{center}

Assume that $C$ has at most one chord, and say $uw$ is not in $G$.
Then $uw$ must be in $\overline{G}$, so it is contained in a region
of $K_{4,4}$, say $acgea$. Thus $u$ and $w$ are both adjacent
to $d$, $f$, $b$, and $h$ in $G$. 

Assume $C$ has chord $vx$. Now in $G$, $d$ and $f$ must be the
on opposite side of cycle $C$ from $b$ and $h$. But then edge $vx$
creates a crossing in $G$, contradicting planarity.

Assume $vx$ is not a chord of $C$, so it is in $\overline{G}$.
If $vx$ is in a region of $\overline{G}$ not containing one of $c$
or $e$ (say $c$) , then $v$ and $x$ are both adjacent to $c$
in $G$. If $vx$ is in another region of $\overline{G}$ containing
both $c$ and $e$, then all vertices of $C$ are adjacent to $d$
and $f$ in $G$. Either way, this creates a crossing in $G$, since
all vertices of $B$ are on the same side of cycle $C$.

Assume $C$ has two chords, so its vertices induce $K_{4}$. Theorem
\ref{thm:336} and Lemma \ref{lem:435} show that no region of the
$K_{4}$ contains 3 or 4 vertices. Thus all four regions contain exactly
two vertices. Each adjacent pair in $V\left(G-C\right)$ and its surrounding
triangle in $K_{4}$ induce $P_{3}+\overline{K}_{2}$. Since $\Delta\left(G\right)\leq8$,
summing degrees shows that each vertex in $V\left(C\right)$ has degree
8. Thus there are two adjacent vertices inside $vwx$, not adjacent
to $u$, contracting Lemma \ref{lem:83}.
\end{proof}

\section{Computer Search}

Embedding algorithms for the torus are complex. Planar graphs are
characterized by only two forbidden minors, $K_{5}$ and $K_{3,3}$
\cite{Wagner 1937}. However, for the torus there are many more\textendash Myrvold
and Woodcock found over 250000 minors by computational search \cite{Myrvold/Woodcock 2018}.
Some previous attempts to devise algorithms for toroidal embedding
have been erroneous, as shown by Myrvold \cite{Myrvold/Kocay 2011}.
There is currently no error-free implementation of a linear-time toroidal
embedding algorithm.

Among the $7595$ plane triangulations, there are $3476$ graphs with
maximum degree greater than $8$ corresponding to Lemma 7. There are
$256$ graphs with an isolated set of three vertices all not adjacent
to a vertex of degree $8$. The filtered set of embeddings are listed
in a text file named \texttt{tri\_isolated\_3.txt} included in the
GitHub repository \cite{Campbell 2025}.

The first computational search would have completed in about $30$
minutes on a laptop with Intel \texttt{i7-9750H} processor with full
CPU utilization. We filtered out plane triangulations that have $\Delta\ge9$
and checked for any complement that had a torus embedding; there were
none. 

Next, we considered all possible ways to remove exactly one edge from
the complement of a plane triangulation with $12$ vertices. The search
took about $4$ days $8$ hours and $24$ minutes, with none found.
For each complement, one edge was chosen to be removed, then the program
searched for a torus embedding, and output any found.

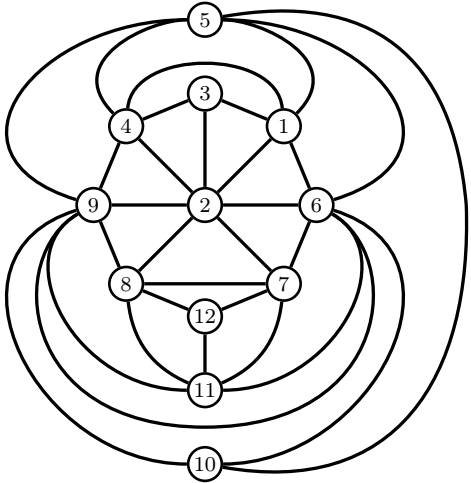
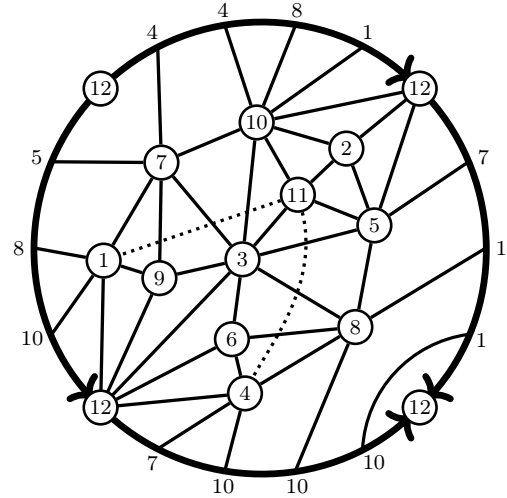
\begin{figure} [hbt!]
\raggedright
\vskip 0.5cm 
\begin{subfigure}[t]{0.48\linewidth} 
\hspace{-1cm}
\begin{tikzpicture}      
\node (2) at (0, 0) {2};   
\node[right=1cm of 2] (6) {6};   
\node[below right=1cm of 2] (7) {7};   
\node[below=1cm of 2] (12) {12};   
\node[below left=1cm of 2] (8) {8};   
\node[left=1cm of 2] (9) {9};   
\node[above left=1cm of 2] (4) {4};   
\node[above=1cm of 2] (3) {3};   
\node[above right=1cm of 2] (1) {1};   
\node[below=0.5cm of 12] (11) {11};   
\node[below=0.5cm of 11] (10) {10};   
\node[above=0.5cm of 3] (5) {5};      
\draw (2) -- (6);   \draw (2) -- (7);   \draw (2) -- (8);   \draw (2) -- (9);   \draw (2) -- (4);   \draw (2) -- (3);   \draw (2) -- (1);   \draw (6) -- (7);   \draw (7) -- (12);   \draw (7) -- (8);   \draw (12) -- (8);   \draw (8) -- (9);   \draw (9) -- (4);   \draw (4) -- (3);   \draw (3) -- (1);   \draw (1) -- (6);      
\coordinate[left=2cm of 5] (5b);   
\coordinate[above left=0.5cm and 2cm of 9] (9b);   
\coordinate[right=2cm of 5] (5c);   
\coordinate[left=0.5cm of 5] (5d);   
\coordinate[right=0.5cm of 5] (5e);   
\coordinate[above right=0.5cm and 2cm of 6] (6b);      
\coordinate[above left=0.5cm and 1cm of 3] (3b);   
\coordinate[above left=0.5cm and 1cm of 9] (9c);   
\coordinate[above right=0.5cm and 1cm of 3] (3c);   
\coordinate[above right=0.5cm and 1cm of 6] (6c);      
\coordinate[left=1.5cm of 11] (11b);   
\coordinate[below left=0.5cm and 1cm of 9] (9d);   
\coordinate[right=1.5cm of 11] (11c);   
\coordinate[below right=0.5cm and 1cm of 6] (6d);   
\coordinate[above left=1cm of 4] (4b);   
\coordinate[above right=1cm of 1] (1b);      
\draw[out=85,in=95] (4) to (1);   
\draw[controls=(9b) and (5b)] (9) to (5);   
\draw[controls=(4b) and (5d)] (4) to (5);   
\draw[controls=(6b) and (5c)] (6) to (5);   \draw[controls=(1b) and (5e)] (1) to (5);      
\draw (12) -- (11);   
\draw[bend left] (11) to (8);   
\draw[bend right] (11) to (7);   
\draw[controls=(9d) and (11b)] (9) to (11);   \draw[controls=(6d) and (11c)] (6) to (11);      \coordinate[left=2cm of 10] (10b);   
\coordinate[below left=0.5cm and 2cm of 9] (9e);   
\coordinate[right=2cm of 10] (10c);   
\coordinate[below right=0.5cm and 2cm of 6] (6e);      
\coordinate[below=0.25cm of 11] (11a);   \coordinate[left=2.5cm of 11a] (11b);   \coordinate[right=2.5cm of 11a] (11c);   
\draw[controls=(9d) and (11b)] (9) to (11a);   
\draw[controls=(6d) and (11c)] (6) to (11a);   
\draw[controls=(9e) and (10b)] (9) to (10);   
\draw[controls=(6e) and (10c)] (6) to (10);      \coordinate[right=1.75cm of 6] (6f);      \draw[out=10,in=95,looseness=1.25] (5) to (6f);   \draw[out=350,in=272,looseness=1.25] (10) to (6f);
\end{tikzpicture} 
\caption{The 5525$^{\text{th}}$ plane triangulation listed in the Combinatorial Object Server.} 
\label{fig:tri-comp2641} 
\end{subfigure} 
\hspace{0.7cm} 
\begin{subfigure}[t]{0.45\linewidth}   
\raggedright   
\begin{tikzpicture}[scale=0.03]     
\def\vertexscale{1.00}     
\def\labelscale{1.00}     
\node [circle,black,draw,scale=\vertexscale] (1) at (-69.35368,-5.56422) {1};     
\node [circle,black,draw,scale=\vertexscale] (2) at (38.20817,44.11726) {2};     
\node [circle,black,draw,scale=\vertexscale] (3) at (-7.82099,-4.99685) {3};     
\node [circle,black,draw,scale=\vertexscale] (4) at (-6.71836,-64.45778) {4};     
\node [circle,black,draw,scale=\vertexscale] (5) at (50.63726,10.05312) {5};     
\node [circle,black,draw,scale=\vertexscale] (6) at (-12.51113,-40.42641) {6};     
\node [circle,black,draw,scale=\vertexscale] (7) at (-43.75057,37.91591) {7};     
\node [circle,black,draw,scale=\vertexscale] (8) at (42.20589,-35.02457) {8};     
\node [circle,black,draw,scale=\vertexscale] (9) at (-44.64359,-13.33567) {9};     
\node [circle,black,draw,scale=\vertexscale] (10) at (-1.55069,55.71368) {10};     
\node [circle,black,draw,scale=\vertexscale] (11) at (16.71879,23.64390) {11};     \tkzDefPoint(-70.71068,-70.71068){12}     \tkzDefPoint(-92.38795,38.26834){13}     \tkzDefPoint(-100.00000,-0.00000){14}     \tkzDefPoint(-92.38795,-38.26834){15}     \tkzDefPoint(-45.39905,-89.10065){16}     \tkzDefPoint(-15.64345,-98.76883){17}     \tkzDefPoint(15.64345,-98.76883){18}     \tkzDefPoint(45.39905,-89.10065){19}     \tkzDefPoint(92.38795,38.26834){20}     \tkzDefPoint(100.00000,0.00000){21}     \tkzDefPoint(92.38795,-38.26834){22}     \tkzDefPoint(-45.39905,89.10065){23}     \tkzDefPoint(-15.64345,98.76883){24}     \tkzDefPoint(15.64345,98.76883){25}     \tkzDefPoint(45.39905,89.10065){26}     \tkzDefPoint(-70.71068,70.71068){27}     \tkzDefPoint(70.71068,70.71068){28}     \tkzDefPoint(70.71068,-70.71068){29}     
\draw [black] (1) to (9);     
\draw [black] (1) to (12);     
\draw [black] (1) to (15);     
\node [draw=none,fill=none,scale=\labelscale] () at (-101.00735,-40.18176) {10};     
\draw [black] (1) to (14);     
\node [draw=none,fill=none,scale=\labelscale] () at (-107.00000,-0.00000) {8};     
\draw [black] (1) to (7);     
\draw [black] (2) to (11);     
\draw [black] (2) to (10);     
\draw [black] (2) to (28);     
\draw [black] (2) to (5);     
\draw [black] (3) to (12);     
\draw [black] (3) to (9);     
\draw [black] (3) to (7);     
\draw [black] (3) to (10);     
\draw [black] (3) to (11);     
\draw [black] (3) to (5);     
\draw [black] (3) to (8);     
\draw [black] (3) to (6);     
\draw [black] (4) to (17);     
\node [draw=none,fill=none,scale=\labelscale] () at (-16.42562,-105.70728) {10};     
\draw [black] (4) to (16);     
\node [draw=none,fill=none,scale=\labelscale] () at (-47.66900,-95.55569) {7};     
\draw [black] (4) to (12);     
\draw [black] (4) to (6);     
\draw [black] (4) to (8);     
\draw [black] (5) to (8);     
\draw [black] (5) to (11);     
\draw [black] (5) to (28);     
\draw [black] (5) to (20);     
\node [draw=none,fill=none,scale=\labelscale] () at (99.00735,40.18176) {7};     
\draw [black] (6) to (8);     
\draw [black] (6) to (12);     
\draw [black] (7) to (13);     
\node [draw=none,fill=none,scale=\labelscale] () at (-99.00735,40.18176) {5};     
\draw [black] (7) to (23);     
\node [draw=none,fill=none,scale=\labelscale] () at (-47.66900,95.55569) {4};     
\draw [black] (7) to (10);     
\draw [black] (7) to (9);     
\draw [black] (8) to (21);     
\node [draw=none,fill=none,scale=\labelscale] () at (107.00000,0.00000) {1};     
\draw [black] (8) to (18);     
\node [draw=none,fill=none,scale=\labelscale] () at (16.42562,-105.70728) {10};     
\draw [black] (9) to (12);     
\draw [black] (10) to (24);     
\node [draw=none,fill=none,scale=\labelscale] () at (-16.42562,105.70728) {4};     
\draw [black] (10) to (25);     
\node [draw=none,fill=none,scale=\labelscale] () at (16.42562,105.70728) {8};     
\draw [black] (10) to (26);     
\node [draw=none,fill=none,scale=\labelscale] () at (47.66900,95.55569) {1};     
\draw [black] (10) to (28);     
\draw [black] (10) to (11);     
\draw [black,dotted] (11) to (1);     
\draw [black,dotted] (11) .. controls ++(10,-40) and ++(10,20) .. (4);     
\tkzDefPoint(45.39905,-89.10065){A}     \tkzDefPoint(55,-61.13712){B}     \tkzDefPoint(92.38795,-38.26834){C}     \tkzCircumCenter(A,B,C)\tkzGetPoint{D}     \tkzDrawArc[black,line width=1.2pt](D,C)(A)     
\node [draw=none,fill=none,scale=\labelscale] () at (50.42562,-95) {10};     
\node [draw=none,fill=none,scale=\labelscale] () at (98,-41) {1};     
\tkzDefPoint(-74.15637,-67.08825){A}     \tkzDefPoint(-74.15637,67.08825){B}     \tkzDefPoint(0.0,0.0){C}     
\tkzDrawArc[->,line width=0.9mm](C,B)(A)     \tkzDefPoint(-67.08825,74.15637){A}     \tkzDefPoint(67.08825,74.15637){B}     
\tkzDefPoint(0.0,0.0){C}     
\tkzDrawArc[<-,line width=0.9mm](C,B)(A)     \tkzDefPoint(74.15637,67.08825){A}     \tkzDefPoint(74.15637,-67.08825){B}     \tkzDefPoint(0.0,0.0){C}     
\tkzDrawArc[<-,line width=0.9mm](C,B)(A)     \tkzDefPoint(67.08825,-74.15637){A}     \tkzDefPoint(-67.08825,-74.15637){B}     \tkzDefPoint(0.0,0.0){C}     
\tkzDrawArc[->,line width=0.9mm](C,B)(A)     
\node [circle,black,draw,fill=white,scale=\vertexscale] (12) at (-70.71068,-70.71068) {12};     
\node [circle,black,draw,fill=white,scale=\vertexscale] (27) at (-70.71068,70.71068) {12};     
\node [circle,black,draw,fill=white,scale=\vertexscale] (28) at (70.71068,70.71068) {12};     
\node [circle,black,draw,fill=white,scale=\vertexscale] (29) at (70.71068,-70.71068) {12};   
\end{tikzpicture}   
\caption{A complement embedding listed on page 72 from \cite{Campbell 2025} found by combinatorial search.  Opposite sides of the region boundary are identified to form the torus.} 
\end{subfigure}
\caption{A planar triangulation and its complement embedded on the torus with two omitted edges shown as dotted curves.} 
\end{figure}

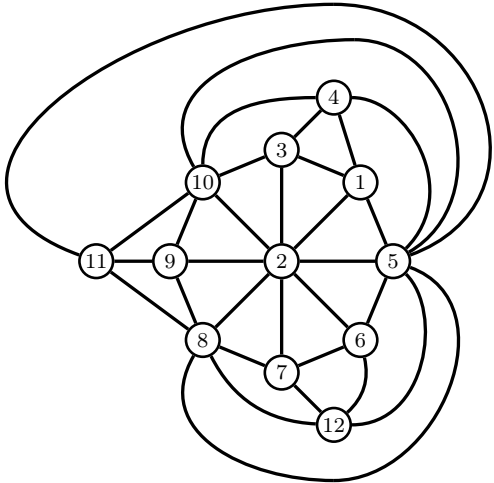
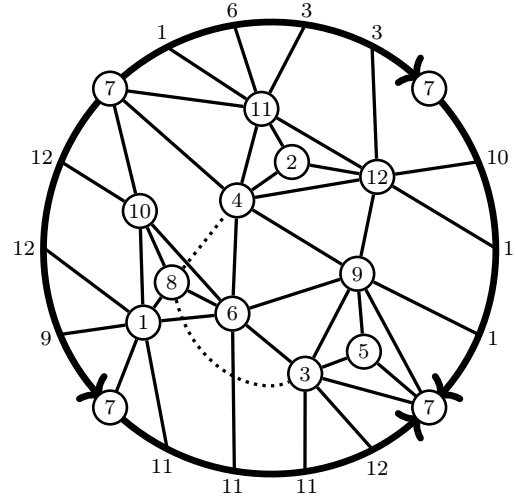
\begin{figure}[hbt!]
\raggedright   
\vskip 0.5cm   
\begin{subfigure}[t]{0.48\linewidth}   
\hspace{-1.5cm} 
\begin{tikzpicture}    
\node (2) at (0, 0) {2};   
\node[right=1cm of 2] (5) {5};   
\node[below right=1cm of 2] (6) {6};   
\node[below=1cm of 2] (7) {7};   
\node[below left=1cm of 2] (8) {8};   
\node[left=1cm of 2] (9) {9};   
\node[above left=1cm of 2] (10) {10};   
\node[above=1cm of 2] (3) {3};   
\node[above right=1cm of 2] (1) {1};   
\node[below right=0.5cm of 7] (12) {12};   
\node[left=0.5cm of 9] (11) {11};   
\node[above right=0.5cm of 3] (4) {4};      
\draw (2) -- (5);   
\draw (2) -- (6);   
\draw (2) -- (7);   
\draw (2) -- (8);   
\draw (2) -- (9);   
\draw (2) -- (10);   
\draw (2) -- (3);   
\draw (2) -- (1);   
\draw (5) -- (6);   
\draw (6) -- (7);   
\draw (7) -- (8);   
\draw (8) -- (9);   
\draw (9) -- (10);   
\draw (10) -- (3);   
\draw (3) -- (1);   
\draw (3) -- (4);   
\draw (1) -- (5);   
\draw (1) -- (4);   
\draw (8) -- (11);   
\draw (9) -- (11);   
\draw (10) -- (11);   
\draw (7) -- (12);      
\coordinate[left=2cm of 4] (4b);   
\coordinate[above left=0.5cm and 2cm of 9] (9b);   
\coordinate[right=0.75cm of 4] (4c);   
\coordinate[above right=0.75cm and 0.75cm of 5] (5b);      
\coordinate[above=0.75cm of 10] (10b);   \coordinate[left=0.75cm of 4] (4d);   
\coordinate[above right=0.5cm and 1cm of 3] (3c);   
\coordinate[above right=0.5cm and 1cm of 5] (5c);      
\coordinate[left=1.5cm of 12] (12b);   
\coordinate[below left=0.5cm and 1cm of 9] (9d);   
\coordinate[right=1cm of 12] (12c);   
\coordinate[below right=0.5cm and 0.5cm of 5] (5d);   \coordinate[below=0.5cm of 12] (12d);   
\coordinate[above right=0.6cm and 0.1cm of 4] (4e);      \draw[out=120,in=180,looseness=1.2] (10) to (4e);   \draw[out=30,in=0,looseness=1.2] (5) to (4e);   \draw[controls=(4d) and (10b)] (4) to (10);   \draw[controls=(5b) and (4c)] (5) to (4);      
\draw[bend left] (12) to (8);   
\draw[bend right] (12) to (6);   
\draw[controls=(5d) and (12c)] (5) to (12);   
\draw[out=340,in=0,looseness=1.2] (5) to (12d);   \draw[out=240,in=180,looseness=1.2] (8) to (12d);      \coordinate[left=2cm of 11] (11b);   
\coordinate[below left=0.5cm and 2cm of 9] (9e);   
\coordinate[right=2cm of 11] (11c);   
\coordinate[below right=0.5cm and 2cm of 5] (5e);   
\coordinate[above=1cm of 4] (4f);      \draw[out=160,in=180,looseness=1.5] (11) to (4f);   \draw[out=20,in=0,looseness=1.5] (5) to (4f);    \end{tikzpicture}   
\caption{The 5557$^{\text{th}}$ planar triangulation listed in the Combinatorial Object Server.}   \label{fig:tri-comp2690} 
\end{subfigure} 
\hspace{0.7cm} 
\begin{subfigure}[t]{0.45\linewidth} 
\raggedright 
\begin{tikzpicture}[scale=0.03]   
\def\vertexscale{1.00}   
\def\labelscale{1.00}   
\node [circle,black,draw,scale=\vertexscale] (1) at (-55.99721,-33.04590) {1};   
\node [circle,black,draw,scale=\vertexscale] (2) at (9.86435,38.15167) {2};   
\node [circle,black,draw,scale=\vertexscale] (3) at (15.73408,-55.73302) {3};   
\node [circle,black,draw,scale=\vertexscale] (4) at (-14.40474,21.12494) {4};   
\node [circle,black,draw,scale=\vertexscale] (5) at (41.76426,-45.95004) {5};   
\node [circle,black,draw,scale=\vertexscale] (6) at (-16.50274,-29.08913) {6};   \tkzDefPoint(-70.71068,-70.71068){7}   
\node [circle,black,draw,scale=\vertexscale] (8) at (-43.35669,-15.22572) {8};   
\node [circle,black,draw,scale=\vertexscale] (9) at (38.84115,-11.40870) {9};   
\node [circle,black,draw,scale=\vertexscale] (10) at (-57.53257,16.43407) {10};   
\node [circle,black,draw,scale=\vertexscale] (11) at (-3.66971,61.74070) {11};   
\node [circle,black,draw,scale=\vertexscale] (12) at (47.67232,31.58414) {12};   
\tkzDefPoint(-92.38795,38.26834){13}   \tkzDefPoint(-100.00000,-0.00000){14}   \tkzDefPoint(-92.38795,-38.26834){15}   \tkzDefPoint(45.39905,-89.10065){16}   \tkzDefPoint(15.64345,-98.76883){17}   \tkzDefPoint(-15.64345,-98.76883){18}   \tkzDefPoint(-45.39905,-89.10065){19}   \tkzDefPoint(92.38795,38.26834){20}   \tkzDefPoint(100.00000,0.00000){21}   \tkzDefPoint(92.38795,-38.26834){22}   \tkzDefPoint(-45.39905,89.10065){23}   \tkzDefPoint(-15.64345,98.76883){24}   \tkzDefPoint(15.64345,98.76883){25}   \tkzDefPoint(45.39905,89.10065){26}   \tkzDefPoint(-70.71068,70.71068){27}   \tkzDefPoint(70.71068,70.71068){28}   \tkzDefPoint(70.71068,-70.71068){29}   
\draw [black] (1) to (8);   
\draw [black] (1) to (6);   
\draw [black] (1) to (19);   
\node [draw=none,fill=none,scale=\labelscale] () at (-47.66900,-96.55569) {11};   
\draw [black] (1) to (7);   
\draw [black] (1) to (15);   
\node [draw=none,fill=none,scale=\labelscale] () at (-99.00735,-40.18176) {9};   
\draw [black] (1) to (14);   
\node [draw=none,fill=none,scale=\labelscale] () at (-109.00000,-0.00000) {12};   
\draw [black] (1) to (10);   
\draw [black] (2) to (4);   
\draw [black] (2) to (11);   
\draw [black] (2) to (12);   
\draw [black] (3) to (16);   
\node [draw=none,fill=none,scale=\labelscale] () at (47.66900,-97.55569) {12};   
\draw [black] (3) to (17);   
\node [draw=none,fill=none,scale=\labelscale] () at (16.42562,-105.70728) {11};   
\draw [black] (3) to (6);   
\draw [black] (3) to (9);   
\draw [black] (3) to (5);   
\draw [black] (3) to (29);   
\draw [black] (4) to (12);   
\draw [black] (4) to (9);   
\draw [black] (4) to (6);   
\draw [black] (4) to (27);   
\draw [black] (4) to (11);   
\draw [black] (5) to (9);   
\draw [black] (5) to (29);   
\draw [black] (6) to (9);   
\draw [black] (6) to (18);   
\node [draw=none,fill=none,scale=\labelscale] () at (-16.42562,-105.70728) {11};   
\draw [black] (6) to (8);   
\draw [black] (6) to (10);   
\draw [black,dotted] (8) .. controls ++(10,-40) and ++(-20,-10) .. (3);   
\draw [black,dotted] (8) to (4);   
\draw [black] (8) to (10);   
\draw [black] (9) to (12);   
\draw [black] (9) to (22);   
\node [draw=none,fill=none,scale=\labelscale] () at (99.00735,-40.18176) {1};   
\draw [black] (9) to (29);   
\draw [black] (10) to (13);   
\node [draw=none,fill=none,scale=\labelscale] () at (-101.00735,40.18176) {12};   
\draw [black] (10) to (27);   
\draw [black] (11) to (27);   
\draw [black] (11) to (23);   
\node [draw=none,fill=none,scale=\labelscale] () at (-47.66900,95.55569) {1};   
\draw [black] (11) to (24);   
\node [draw=none,fill=none,scale=\labelscale] () at (-16.42562,105.70728) {6};   
\draw [black] (11) to (25);   
\node [draw=none,fill=none,scale=\labelscale] () at (16.42562,105.70728) {3};   
\draw [black] (11) to (12);   
\draw [black] (12) to (26);   
\node [draw=none,fill=none,scale=\labelscale] () at (47.66900,95.55569) {3};   
\draw [black] (12) to (20);   
\node [draw=none,fill=none,scale=\labelscale] () at (101.00735,40.18176) {10};   
\draw [black] (12) to (21);   
\node [draw=none,fill=none,scale=\labelscale] () at (106.00000,0.00000) {1};   
\tkzDefPoint(-74.15637,-67.08825){A}   \tkzDefPoint(-74.15637,67.08825){B}   
\tkzDefPoint(0.0,0.0){C}   
\tkzDrawArc[->,line width=0.9mm](C,B)(A)   \tkzDefPoint(-67.08825,74.15637){A}   \tkzDefPoint(67.08825,74.15637){B}   
\tkzDefPoint(0.0,0.0){C}   
\tkzDrawArc[<-,line width=0.9mm](C,B)(A)   \tkzDefPoint(74.15637,67.08825){A}   \tkzDefPoint(74.15637,-67.08825){B}   
\tkzDefPoint(0.0,0.0){C}   
\tkzDrawArc[<-,line width=0.9mm](C,B)(A)   \tkzDefPoint(67.08825,-74.15637){A}   \tkzDefPoint(-67.08825,-74.15637){B}   
\tkzDefPoint(0.0,0.0){C}   
\tkzDrawArc[->,line width=0.9mm](C,B)(A)   
\node [circle,black,draw,fill=white,scale=\vertexscale] (7) at (-70.71068,-70.71068) {7};   
\node [circle,black,draw,fill=white,scale=\vertexscale] (27) at (-70.71068,70.71068) {7};   
\node [circle,black,draw,fill=white,scale=\vertexscale] (28) at (70.71068,70.71068) {7};   
\node [circle,black,draw,fill=white,scale=\vertexscale] (29) at (70.71068,-70.71068) {7}; 
\end{tikzpicture} 
\caption{A complement embedding listed on page 91 from \cite{Campbell 2025} found by combinatorial search.  Opposite sides of the region boundary are identified to form the torus.} 
\end{subfigure} 
\caption{A planar triangulation and its complement embedded on the torus with two omitted edges shown as dotted curves.} 
\end{figure}

The computational search for removing all possible pairs of edges
from each complement of each plane triangulation with $12$ vertices
took about $59$ days. For each complement, two edges were chosen
to be removed, then the program searched for a torus embedding, and
output any found. Duplicates were logged in a text file of output.
Once isomorphic embeddings were removed, there were $123$ unique
maximal torus embeddings with $2$ edges removed from the complement
of plane triangulations with $12$ vertices. Through use of \texttt{planedraw}
by Gunnar Brinkmann, the list of $123$ torus embeddings can be seen
as figures in the PDF compiled document available in our GitHub repository
\cite{Campbell 2025} along with a subfolder of images giving corresponding
embeddings of the triangulations of the plane on $12$ vertices. The
images were created using personal \texttt{C/C++} code by Campbell.
Two examples of a triangulation with 12 vertices and its complement
so that two edges can be deleted to embed on the torus are given in
Figures 1 and 2.

For the search algorithm, we did not program it to avoid constructing
duplicate embeddings. Two combinatorial embeddings $G$ and $H$ on
an oriented surface can be considered equivalent (or flip-isomorphic)
if the circular adjacency lists for each vertex of $G$ are reversed
and then equal to the circular adjacency lists of $H$. The reversal
of all adjacency lists represents flipping a clockwise ordering to
a counterclockwise ordering on the surface and vice versa. To reduce
the number of equivalent embeddings found in an algorithm, one can
canonically choose in the circular adjacency list which vertices will
be the first two in any ordering during combinatorial searches. Although
allowing equivalent duplicates increases runtime, it helps to check
correctness where any recursive construction of an embedding matches
at least one duplicate.

For more specific algorithm descriptions that can be applied to either
orientable or nonorientable surfaces, refer to Campbell's dissertation
\cite{Campbell 2017} available online through UVic DSpace (URI link
provided in reference section).

The list of plane triangulations of order $12$ can also be generated
by Sulanke's {\tt Surftri} \cite{Sulanke 2006}, and once compiled,
with the following command ({\tt -a} option for ascii format; {\tt 0}
option required to specify genus):

\begin{verbatim} ./surftri -a 12 0 \end{verbatim}

\noindent To execute the command, one also needs the corresponding
{\tt genus0.alpha} file giving the one irreducible triangulation
of the sphere $K_{4}$ that can either be downloaded from Sulanke's
website or just create the file with its one-line contents: ``{\tt 4 bcd,adc,abd,acb}.''
The program will print one graph on each line of output starting with
each graph's order, followed by adjacency lists separated by comma
given as alphabetic labels of the vertex neighbors. The adjacency
lists can either be treated as all in counterclockwise or all clockwise
order, as needed. {\tt Surftri} outputs the $7595$ triangulation
plane embeddings in less than a second on a basic modern computer.
As isomorphism is a computationally complex task, we do not check
which triangulation matches among those listed in the Combinatorial
Object Server, but we wish to encourage use of multiple resources
to foster a more resilient research community.

Lastly, to aid our research efforts and to double-check any results,
an image collection of the $7595$ plane triangulations was generated
procedurally with a drawing algorithm written in C/C++. Note that
individual images procedurally generated are not currently copyrightable
in Canadian law, and not likely to be copyrightable in United States
law. However, collections of such images can be protected. The collection
is available in GitHub repository and license described therein \cite{Campbell 2025}.

\begin{acknowledgement*}
This paper is dedicated to Art White, who informed the first author
of this problem a decade ago.

Campbell wishes to thank Allan Bickle for awareness of the problem
and their collaboration. 
\end{acknowledgement*}

\end{document}